\documentclass[a4paper,10pt]{article}
\usepackage[utf8]{inputenc}
\usepackage[T1]{fontenc}
\usepackage{lmodern}
\usepackage[colorlinks=true,linkcolor=blue]{hyperref}
\usepackage{amsmath,amssymb,amsthm}
\usepackage{amsfonts}
\usepackage[english]{babel} 
\usepackage{dsfont}
\usepackage{enumerate}
\usepackage{listings}
\usepackage{graphicx}
\usepackage{array}
\usepackage{caption}
\usepackage{subcaption}
\usepackage{xcolor}
\usepackage{pgf,tikz}
\usetikzlibrary{decorations.pathreplacing}
\usepackage{stmaryrd}
\usepackage{xr}

\usepackage[a4paper]{geometry}
\newtheorem{rem}{Remark}[section]
\newtheorem{thm}{Theorem}[section]
\newtheorem{lem}{Lemma}[section]
\newtheorem{propo}{Proposition}[section]

\newtheorem{hypothesis}[thm]{Hypothesis}

\numberwithin{equation}{section}
\newcommand*{\R}{\mathbb{R}} % L'ensemble des réels
\newcommand*{\N}{\mathbb{N}} % L'ensemble des entiers
\newcommand*{\F}{\mathcal{F}} % filtration F
\newcommand*{\PP}{\mathbb{P}}% Probabilité
\newcommand*{\E}{\mathbb{E}}% Espérance
\DeclareGraphicsExtensions{.jpg,.png,.pdf,.bmp,.eps}

\title{Concentration inequalities for Stochastic Differential Equations with additive fractional noise}

\author{Maylis Varvenne\thanks{Institut de Math\'ematiques de Toulouse, Universit\'e Paul Sabatier, 118 route de Narbonne, 31062 Toulouse, France. E-mail: maylis.varvenne@math.univ-toulouse.fr}}

{\scriptsize\date\today}

\begin{document}

\maketitle
\begin{abstract}
In this paper, we establish concentration inequalities both for functionals of the whole solution on an interval $[0,T]$ of an additive SDE driven by a fractional Brownian motion with Hurst parameter $H\in(0,1)$ and for functionals of discrete-time observations of this process. Then, we apply this general result to specific functionals related to discrete and continuous-time occupation measures of the process.
\end{abstract}

\bigskip


\section{Introduction}

In this article, we consider the solution $(Y_t)_{t\geq0}$ of the following $\R^d$-valued Stochastic Differential Equation (SDE) with additive noise:
\begin{equation}\label{eq:sde_intro}
Y_t=x+\int_0^t  b(Y_s) {\rm d}s +  \sigma B_t.
\end{equation}
with $B$ a $d$-dimensional fractional Brownian motion (fBm) with Hurst parameter $H\in(0,1)$.
We are interested in questions of long-time concentration phenomenon of the law of the solution $Y$. A well known way to overcome this type of problem is to prove $L^1$-transportation inequalities. Let us precise what it means.
Let $(E,d)$ be a metric space equipped with a $\sigma$-field $\mathcal{B}$ such that the distance $d$ is $\mathcal{B}\otimes\mathcal{B}$-measurable. Given $p\geqslant1$ and two probability measures $\mu$ and $\nu$ on $E$, the Wasserstein distance is defined by 
$$\mathcal{W}_p(\mu,\nu)=\inf\left(\int\int d(x,y)^p{\rm d}\pi(x,y)\right),$$
where the infimum runs over all the probability measures $\pi$ on $E\times E$ with marginals $\mu$ and $\nu$. The entropy of $\nu$ with respect to $\mu$ is defined by
$$\mathbf{H}(\nu|\mu)=\left\{\begin{array}{ll}
\int\log\left(\frac{{\rm d}\nu}{{\rm d}\mu}\right){\rm d}\nu, & \text{ if }\nu\ll\mu,\\
+\infty & \text{ otherwise.}
\end{array}\right.$$
Then, we say that $\mu$ satisfies an $L^p$-transportation inequality with constant $C\geqslant0$ (noted $\mu\in T_p(C)$) if for any probability measure $\nu$,
\begin{equation}\label{eq:transport_ineq}
\mathcal{W}_p(\mu,\nu)\leqslant\sqrt{2C\mathbf{H}(\nu|\mu)}.
\end{equation}
The concentration of measure is intrinsically linked to the above inequality when $p=1$. This fact was first emphasized by K.Marton \cite{marton1996measure, marton1996bounding}, M.Talagrand \cite{talagrand1996transportation}, Bobkov and Götze \cite{bobkov1999exponential} and amply investigated by M.Ledoux \cite{ledoux2001concentration, ledoux2002concentration}.  Indeed, it can be shown (see \cite{ledoux2001concentration} for a detailed proof) that \eqref{eq:transport_ineq} for $p=1$ is actually equivalent to the following: for any $\mu$-integrable $\alpha$-Lipschitz function $F$ (real valued) we have for all $\lambda\in\R$,
\begin{equation}\label{eq:exp_moment_equiv_transport}
\E\left(\exp\left(\lambda\left(F(X)-\E[F(X)]\right)\right)\right)\leqslant\exp\left(C\alpha^2\frac{\lambda^2}{2}\right)
\end{equation}
with $\mathcal{L}(X)=\mu$.
This upper bound naturally leads to concentration inequalities through the classical Markov inequality. For several years, $L^1$ (and $L^2$ since $T_2(C)$ implies $T_1(C)$) transportation inequalities have then been widely studied and in particular for diffusion processes (see for instance \cite{djellout2004transportation, wu2004talagrand, gourcy2006logarithmic}). \\
For SDE's driven by more general Gaussian processes, S.Riedel established transportation cost inequalities in \cite{riedel2014talagrand} using Rough Path theory. However, his results do not give long-time concentration, which is our focus here.\\
In the setting of fractional noise, T.Guendouzi \cite{guendouzi2012transportation} and B.Saussereau \cite{saussereau2012transportation} have studied transportation inequalities with different metrics in the case where $H\in(1/2,1)$. In particular, B.Saussereau gave an important contribution: he proved $T_1(C)$ and $T_2(C)$ for the law of $(Y_t)_{t\in[0,T]}$ in various settings and he got a result of large-time asymptotics in the case of a contractive drift.
Our first motivation to this work was to get equivalent results in a discrete-time context, i.e. for $\mathcal{L}((Y_{k\Delta})_{1\leqslant k\leqslant n})$ for a given step $\Delta>0$ and then long-time concentration inequalities for the occupation measure, i.e. for $\frac{1}{n}\sum_{k=1}^nf(Y_{k\Delta})$ (where $f$ is a general Lipschitz function real valued). 
Indeed, in a statistical framework we only have access to discrete-time observations of the process $Y$ and such a result could be meaningful in such context. To the best of our knowledge, this type of result is unknown in the fractional setting.\\ % and if we want some convergence results.\\
We first tried to adapt the methods used in \cite{saussereau2012transportation} in several ways as for example: find a distance such that $(y_t)_{t\in[0,T]}\mapsto (y_{k\Delta})_{1\leqslant k\leqslant n}$ is Lipschitz and prove $T_1(C)$ with this metric. But the constants obtained in the $L^1$-transportation inequalities were not sharp enough, so that we couldn't deduce large-time asymptotic as B.Saussereau. \\
In \cite{djellout2004transportation}, H.Djellout, A.Guillin and L.Wu explored transportation inequalities in the diffusive case and both in a continuous and discrete-time setting. In particular, for the discrete-time case, they used a kind of tensorization of the $L^1$ transportation inequality but the Markovian nature of the process was essential. However, they prove $T_1(C)$ through its equivalent formulation \eqref{eq:exp_moment_equiv_transport} and to this end, they apply a decomposition of the functional in \eqref{eq:exp_moment_equiv_transport} into a sum of martingale increments, namely:
$$F(X)-\E[F(X)]=\sum_{k=1}^n\E[F(X)|\mathcal{F}_k]-\E[F(X)|\mathcal{F}_{k-1}]$$ 
with $X=(Y_{k\Delta})_{1\leqslant k\leqslant n}$ and $Y$ is the solution of \eqref{eq:sde_intro} when $B$ is the classical Brownian motion. \\
This decomposition has inspired the approach described in this paper: instead of proving an $L^1$ transportation inequality \eqref{eq:transport_ineq}, we prove its equivalent formulation \eqref{eq:exp_moment_equiv_transport} by using a similar decomposition and the series expansion of the exponential function. % This is result a bit weaker since we do not get a constant equal to $1$ behind the exponential function in \eqref{eq:exp_moment_equiv_transport}.
Through this strategy, we prove several results under an assumption of contractivity on the drift term $b$ in \eqref{eq:sde_intro}. First, in a discrete-time setting, we work in the space $(\R^d)^n$ endowed with the $L^1$ metric and we show that for any $\alpha$-Lipschitz functional $F:(\R^d)^n\to\R$ and for any $\lambda>0$,
$$\E\left(\exp\left(\lambda\left(F(X)-\E[F(X)]\right)\right)\right)\leqslant \exp\left(C\alpha^2\lambda^2n^{2H\vee1}\right)$$
with $X=(Y_{k\Delta})_{1\leqslant k\leqslant n}$. In a similar way, we consider the space of continuous functions $\mathcal{C}([0,T],\R^d)$ endowed with the $L^1$ metric and we prove that for any $\alpha$-Lipschitz functional $\tilde{F}:\mathcal{C}([0,T],\R^d)\to\R$ and for any $\lambda>0$,
$$\E\left(\exp\left(\lambda\left(\tilde{F}(X)-\E[\tilde{F}(X)]\right)\right)\right)\leqslant \exp\left(C\alpha^2\lambda^2T^{2H\vee1}\right)$$
with $X=(Y_{t})_{t\in[0,T]}$. From these inequalities, we deduce some general concentration inequalities and large-time asymptotics for occupation measures. Let us note that 
%these results are a bit weaker than transportation inequalities since we do not have a constant equal to $1$ behind the exponential function in the right hand term as in \eqref{eq:exp_moment_equiv_transport}. However, 
we have no restriction on the Hurst parameter $H$ and we retrieve the results given by B.Saussereau for $H\in(1/2,1)$ in a continuous setting and also the result given in \cite{djellout2004transportation} for $H=1/2$, namely for diffusion.\\

The paper is organised as follows. In the next section, we describe the assumptions on the drift term and we state the general theorem about concentration, namely Theorem \ref{thm:main_result}. Then, in Subsection \ref{subsection:occupation_measure}, we apply this result to specific functionals related to the occupation measures (both in a discrete-time and in a continuous-time framework). Section \ref{section:sketch_of_proof} outlines our strategy of proof which is fulfilled in Sections \ref{section:sum_martingale} and \ref{section:moments}. 

\section{Setting and main results}

\subsection{Notations}

The usual scalar product on $\R^d$ is denoted by $\langle~,~\rangle$ and $|~.~|$ stands either for the Euclidean norm on
$\R^d$ or the absolute value on $\R$. We denote by $\mathcal{M}_d(\R)$ the space of real matrices of size $d\times d$. For a given $n\in\N^*$ and $(x,y)\in\left(\R^d\right)^n\times\left(\R^d\right)^n$, we denote by $d_n$ the following $L^1$-distance: 
\begin{equation}\label{eq:def_dn}
d_n(x,y):=\sum_{k=1}^n|x_i-y_i|.
\end{equation} 
Analogeously, for a given $T>0$ and $(x,y)\in\mathcal{C}\left([0,T],\R^d\right)\times\mathcal{C}\left([0,T],\R^d\right)$, we denote by $d_T$ the classical $L^1$-distance: 
\begin{equation}\label{eq:def_dT}
d_T(x,y):=\int_0^T|x_t-y_t|{\rm d}t.
\end{equation}
Let $F: (E,d_E)\to(E',d_{E'})$ be a Lipschiz function between two metric spaces, we denote by $$\|F\|_{\rm Lip}:=\sup\limits_{x\neq y}\frac{d_{E'}(F(x),F(y))}{d_E(x,y)}$$ its Lipschitz norm.\\
Let $w,\tilde{w}\in\mathcal{C}(\R_+,\R^d)$, let $a,b,c\in\R_+$ such that $a<b<c$. Then, we define 
\begin{equation}
w_{[a,b]}\sqcup \tilde{w}_{[b,c]} (t):=\left\{\begin{array}{ll}w(t) & \text{ if } a\leqslant t\leqslant b\\
\tilde{w}(t)  & \text{ if } b<t\leqslant c.\\

\end{array}\right.
\end{equation}
\subsection{Assumptions and general result}

Let $B$ be a $d$-dimensional fractional Brownian motion (fBm) with Hurst parameter $H\in\left(0,1\right)$ defined on $(\Omega,\mathcal{F},\PP)$ and transferred from a $d$-dimensional Brownian motion $W$ through the Volterra representation (see e.g. \cite{decreusefond1999stochastic, carmona2003stochastic})
\begin{equation}\label{eq:volterra_rep}
\forall t\in\R_+,\quad B_t=\int_{0}^t K^{}_H(t,s){\rm d}W_s,
\end{equation}
with
\begin{equation}\label{eq:kernel_K_H}
K^{}_H(t,s):=c_H\left[\frac{t^{H-\frac{1}{2}}}{s^{H-\frac{1}{2}}}(t-s)^{H-\frac1{2}}-\left(H-\frac1{2}\right)\int_s^t\frac{u^{H-\frac{3}{2}}}{s^{H-\frac{1}{2}}}(u-s)^{H-\frac1{2}}{\rm d}u\right].
\end{equation}
In the sequel, the distribution of $W$ will be denoted by $\PP_W$.\\

We consider the following  $\R^d$-valued stochastic differential equation driven by $B$:
\begin{equation}\label{eq:sde}
Y_t=x+\int_0^t  b(Y_s) {\rm d}s +  \sigma B_t , \qquad t\geqslant0.
\end{equation}
Here $x \in \mathbb{R}^d$ is a given initial condition, $B$ is the aformentioned fractional Brownian motion and $\sigma\in\mathcal{M}_d(\R)$.\\

We are working under the following assumption :
\begin{hypothesis}\label{hyp:b-coercive} We have $b \in \mathcal{C} (\mathbb{R}^d; \mathbb{R}^d)$ and 
there exist constants $\alpha, L>0$ such that:

\noindent
\emph{(i)} For every $x,y\in\R^{d}$,
\begin{equation*}
\langle  b(x)-b(y), \, x-y\rangle \leq  - \alpha |x-y|^2.
\end{equation*}
\noindent
\emph{(ii)} For every $x,y\in\R^{d}$,
\begin{equation*} 
|b(x)-b(y)|\leq L|x-y|.
\end{equation*}
\end{hypothesis}
\begin{rem} $\rhd$ Since $b$ is Lipschitz and $\sigma$ is constant, $Y$ in \eqref{eq:sde} denotes the unique strong solution.\\
$\rhd$ This contractivity assumption on the drift term is quite usual to get long-time concentration results (see \cite{djellout2004transportation,saussereau2012transportation} for instance). At this stage, a more general framework seems elusive.
\end{rem}

Let $T>0$ and $n\in\N^*$.
Let $F:\left((\R^d)^n,d_n\right)\to(\R,|\cdot|)$ and $\tilde{F}:\left(\mathcal{C}\left([0,T],\R^d\right),d_T\right)\to(\R,|\cdot|)$ be two Lipschitz functions and set  
\begin{equation}\label{eq:def_FY_FtildeY}
F_Y:=F(Y_{t_1},\dots,Y_{t_n})\quad\text{ and }\quad\tilde{F}_Y=\tilde{F}((Y_t)_{t\in[0,T]})
\end{equation}
with $0<\Delta=t_1<\dots<t_n$ and $t_{k+1}-t_k=\Delta$ for a given $\Delta>0$.\\

We are now in position to state our results for general functionals $F$ and $\tilde{F}$. First, we prove a result on the exponential moments of $F_Y$ and $\tilde{F}_Y$ which is crucial to get Theorem \ref{thm:main_result}.

\begin{propo}\label{prop:exp_moment_F_Ftilde}
Let $H\in(0,1)$ and $\Delta>0$. Let $n\in\N^*$, $T\geqslant1$ and $d_n,d_T$ be the metrics defined respectively by \eqref{eq:def_dn} and \eqref{eq:def_dT}. Then,
\begin{itemize}
\item[(i)] there exist $C_{H,\Delta,\sigma}>0$ such that for all Lipschitz functions $F:\left((\R^d)^n,d_n\right)\to(\R,|\cdot|)$ and for all $\lambda>0$,
\begin{equation}\label{eq:moment_F}
\E\left[\exp\left(\lambda(F_Y-\E[F_Y])\right)\right]\leqslant \exp\left(C_{H,\Delta,\sigma}\|F\|^2_{\rm Lip}\lambda^2n^{2H\vee1}\right).
\end{equation}
\item[(ii)]there exist $\tilde{C}_{H,\sigma}>0$ such that for all Lipschitz functions $\tilde{F}:\left(\mathcal{C}\left([0,T],\R^d\right),d_T\right)\to(\R,|\cdot|)$ and for all $\lambda>0$,
\begin{equation}\label{eq:moment_Ftilde}
\E\left[\exp\left(\lambda(\tilde{F}_Y-\E[\tilde{F}_Y])\right)\right]\leqslant \exp\left(\tilde{C}_{H,\sigma}\|\tilde{F}\|^2_{\rm Lip}\lambda^2T^{2H\vee1}\right).
\end{equation}
\end{itemize}
\end{propo}

\begin{rem} Let us note that this proposition is actually equivalent to $L^1$-transportation inequalities as mentionned in the introduction. More precisely, item $(i)$ is equivalent to $\mathcal{L}((Y_{t_k})_{1\leq k\leq n})\in~T_1(2C_{H,\Delta,\sigma}n^{2H\vee1})$ for the metric $d_n$ and item $(ii)$ is equivalent to $\mathcal{L}((Y_{t})_{t\in[0,T]})\in T_1(2\tilde{C}_{H,\sigma}T^{2H\vee1})$ for the metric $d_T$.
\end{rem}

From Proposition \ref{prop:exp_moment_F_Ftilde}, we deduce the following concentration inequalities:
\begin{thm}\label{thm:main_result}Let $H\in(0,1)$ and $\Delta>0$. Let $n\in\N^*$, $T\geqslant1$ and $d_n,d_T$ be the metrics defined respectively by \eqref{eq:def_dn} and \eqref{eq:def_dT}. Then,
\begin{itemize}
\item[(i)] there exist $C_{H,\Delta,\sigma}>0$ such that for all Lipschitz functions $F:\left((\R^d)^n,d_n\right)\to(\R,|\cdot|)$ and for all $r\geqslant0$,
\begin{equation}\label{eq:concentration_F}
\PP\left(F_Y-\E[F_Y]>r\right)\leqslant \exp\left(-\frac{r^2}{4C_{H,\Delta,\sigma}\|F\|^2_{\rm Lip}n^{2H\vee1}}\right).
\end{equation}
\item[(ii)]there exist $\tilde{C}_{H,\sigma}>0$ such that for all Lipschitz functions $\tilde{F}:\left(\mathcal{C}\left([0,T],\R^d\right),d_T\right)\to(\R,|\cdot|)$ and for all $r\geqslant0$,
\begin{equation}\label{eq:concentration_Ftilde}
\PP\left(\tilde{F}_Y-\E[\tilde{F}_Y])>r\right)\leqslant \exp\left(-\frac{r^2}{4\tilde{C}_{H,\sigma}\|\tilde{F}\|^2_{\rm Lip}T^{2H\vee1}}\right).
\end{equation}
\end{itemize}
\end{thm}
\begin{proof}We use Markov inequality and Proposition \ref{prop:exp_moment_F_Ftilde}. Then, we optimize in $\lambda$ to get the result.
\end{proof}

\begin{rem} $\rhd$ The dependency on the Lipschitz constant of $F$ and $\tilde{F}$ is essential since they may depend on $n$ and $T$. Hence, if they decrease fast than $n^{-2H\vee1}$ and $T^{-2H\vee1}$, we get large time concentration inequalities.\\
$\rhd$ Let us note that this result remains true if the noise process in \eqref{eq:sde} is replaced by the Liouville fractional Brownian $\tilde{B}$ motion which has the following simpler representation: $\tilde{B}_t=\int_0^t (t-s)^{H-1/2}{\rm d}W_s$. The proof follows exactly the same lines.
\end{rem}
In the following subsection, we outline our main application of Theorem \ref{thm:main_result} for which long time concentration holds.
\subsection{Long time concentration inequalities for occupation measures}\label{subsection:occupation_measure}

We now apply our general result to specific functionals to get the following theorem.
\begin{thm} Let $H\in(0,1)$ and $\Delta>0$. Let $n\in\N^*$ and $T\geqslant1$. Then,
\begin{itemize}
\item[(i)] there exist $C_{H,\Delta,\sigma}>0$ such that for all Lipschitz functions $f:\left(\R^d,|\cdot|\right)\to(\R,|\cdot|)$ and for all $r\geqslant0$,
\begin{equation}\label{eq:occ_measure_moment_F}
\PP\left(\frac{1}{n}\sum_{k=1}^nf(Y_{t_k})-\E[f(Y_{t_k})]>r\right)\leqslant \exp\left(-\frac{r^2n^{2-(2H\vee1)}}{4C_{H,\Delta,\sigma}\|f\|^2_{\rm Lip}}\right).
\end{equation}
\item[(ii)]there exist $\tilde{C}_{H,\sigma}>0$ such that for all Lipschitz functions $f:\left(\R^d,|\cdot|\right)\to(\R,|\cdot|)$ and for all $r\geqslant0$,
\begin{equation}\label{eq:occ_measure_moment_Ftilde}
\PP\left(\frac{1}{T}\int_{0}^T(f(Y_t)-\E[f(Y_t)]){\rm d}t>r\right)\leqslant \exp\left(-\frac{r^2T^{2-(2H\vee1)}}{4\tilde{C}_{H,\sigma}\|f\|^2_{\rm Lip}}\right).
\end{equation}
\end{itemize}
\end{thm}

\begin{proof}
We apply Theorem \ref{thm:main_result} with the following functions $F$ and $\tilde{F}$:
\begin{equation*}
\forall x\in\left(\R^d\right)^n, \quad F(x)=\frac{1}{n}\sum_{k=1}^{n}f(x_i)
\end{equation*}
and
\begin{equation*}
\forall x\in\mathcal{C}\left([0,T],\R^d\right), \quad F(x)=\frac{1}{T}\int_{0}^{T}f(x_t){\rm d}t
\end{equation*}	
which are respectively $\frac{\|f\|_{\rm Lip}}{n}$-Lipschitz with respect to $d_n$ (defined by \eqref{eq:def_dn}) and $\frac{\|f\|_{\rm Lip}}{T}$-Lipschitz with respect to $d_T$ (defined by \eqref{eq:def_dT}).
\end{proof}

\section{Sketch of proof}\label{section:sketch_of_proof}

Recall that $F_Y$ and $\tilde{F}_Y$ are defined by \eqref{eq:def_FY_FtildeY}.
The key element to get the bound \eqref{eq:moment_F} and \eqref{eq:moment_Ftilde} is to decompose $F_Y$ and $\tilde{F}_Y$ into a sum of martingale increments as follows. Let $(\mathcal{F}_t)_{t\geqslant0}$ be the natural filtration associated to the standard Brownian motion $W$ from which the fBm is derived through \eqref{eq:volterra_rep}. For all $k\in\N$, set 
\begin{equation}
M_k:=\E[F_Y~|~\mathcal{F}_{t_k}]\quad\text{ and }\quad \tilde{M}_k:=\E[\tilde{F}_Y~|~\mathcal{F}_{k}].
\end{equation}
With these definitions, we have:
\begin{align}\label{eq:decomposition_sum_martingale_inc}
F_Y-\E[F_Y]=M_n=\sum_{k=1}^nM_k-M_{k-1}
\quad\text{ and }\quad\tilde{F}_Y-\E[\tilde{F}_Y]=\tilde{M}_{\lceil T\rceil}=\sum_{k=1}^{\lceil T\rceil}\tilde{M}_k-\tilde{M}_{k-1}
\end{align}
where $\lceil T\rceil$ denotes the least integer greater than or equal to $T$.\\

With this decomposition in hand, we first estimate the conditional exponential moments of the martingale increments $M_k-M_{k-1}$ and $\tilde{M}_k-\tilde{M}_{k-1}$ to get Proposition \ref{prop:exp_moment_F_Ftilde}. This is the purpose of Proposition \ref{prop:exp_moments_martingale_inc} for which the proof is based on the following lemma:
\begin{lem}\label{lem:moments_to_exponential_moments} Let $X$ be a centered real valued random variable such that for all $p\geq2$, there exist $C,\zeta>0$ such that
$$\E[|X|^p]\leq C\zeta^{p/2}p\Gamma\left(\frac{p}{2}\right).$$
Then for all $\lambda>0$,
$$\E[e^{\lambda X}]\leq e^{2C'\zeta\lambda^2}$$
with $C'=1\vee C$.
\end{lem}

\begin{proof} Since $X$ is centered, by using the series expansion of the exponential function, we have:
\begin{align}\label{eq:proof_exp_moment_X}
\E\left[\exp\left(\lambda X\right)\right]\leqslant 1 + C \sum_{p=2}^{+\infty}\frac{\lambda^p \zeta^{\frac{p}{2}}p\Gamma\left(\frac{p}{2}\right)}{p!}\leqslant 1+ \sum_{p=2}^{+\infty}\frac{\lambda^p (C'\zeta)^{\frac{p}{2}}p\Gamma\left(\frac{p}{2}\right)}{p!}
\end{align}
with $C'=1\vee C$.
We set $t^2=\lambda^2C'\zeta$, then
\begin{align*}
1+ \sum_{p=2}^{+\infty}\frac{(t^2)^{\frac{p}{2}}p\Gamma\left(\frac{p}{2}\right)}{p!}
&= 1+ \sum_{p=1}^{+\infty}\frac{(t^2)^{p}2p\Gamma\left(p\right)}{(2p)!} + \sum_{p=1}^{+\infty}\frac{(t^2)^{p+\frac{1}{2}}(2p+1)\Gamma\left(p+\frac{1}{2}\right)}{(2p+1)!}\\
&=1+ 2\sum_{p=1}^{+\infty}\frac{(t^2)^{p}\Gamma\left(p+1\right)}{(2p)!} + |t|\sum_{p=1}^{+\infty}\frac{(t^2)^{p}\Gamma\left(p+\frac{1}{2}\right)}{(2p)!}\\
&\leqslant1+(2+|t|)\sum_{p=1}^{+\infty}\frac{(t^2)^{p}p!}{(2p)!}\\
&\leqslant 1+\left(1+\frac{|t|}{2}\right)(e^{t^2}-1)\quad \text{ since }\quad 2(p!)^2\leqslant(2p)!~.%\\
%&\leqslant e^{2t^2}.
\end{align*}
Since for all $t\in\R$, $\frac{|t|}{2}\leqslant e^{t^2}$, we get $\frac{|t|}{2}(e^{t^2}-1)\leqslant e^{t^2}(e^{t^2}-1)$ which is equivalent to $$1+\left(1+\frac{|t|}{2}\right)(e^{t^2}-1)\leqslant e^{2t^2},$$ so that:
$$1+ \sum_{p=2}^{+\infty}\frac{(t^2)^{\frac{p}{2}}p\Gamma\left(\frac{p}{2}\right)}{p!}\leqslant e^{2t^2}.$$
Hence, we have in \eqref{eq:proof_exp_moment_X}:
\begin{equation*}
\E\left[\exp\left(\lambda X\right)\right]\leqslant \exp\left(2\lambda^2\zeta C'\right)
\end{equation*}
which concludes the proof.

\end{proof}

\begin{rem} The previous proof follows the proof of Lemma 1.5 in Chapter 1 of \cite{rigollet2017high}. We chose to give the details here since this step is crucial to get our main results.
\end{rem}

Finally, the end of the proof of Proposition \ref{prop:exp_moment_F_Ftilde} $(i)$ is based on the following implication:
if there exists a deterministic sequence $(u_k)$ such that
$$\E\left[\left.e^{\lambda(M_k-M_{k-1})}\right|\mathcal{F}_{k-1}\right]\leq e^{\lambda^2u_k},$$
then
$$\E\left[e^{\lambda M_n}\right]=\E\left[e^{\lambda M_{n-1}}\E\left[\left.e^{\lambda(M_n-M_{n-1})}\right|\mathcal{F}_{n-1}\right]\right]\leq \exp\left(\lambda^2u_n\right)\E\left[e^{\lambda M_{n-1}}\right]$$
so that 
$$\E\left[e^{\lambda M_n}\right]\leq \exp\left(\lambda^2\sum_{k=1}^{n}u_k\right).$$
The same arguments are used for item $(ii)$ of Proposition \ref{prop:exp_moment_F_Ftilde}.
\vspace{2mm}

Sections \ref{section:sum_martingale} and \ref{section:moments} are devoted to the proof of Proposition \ref{prop:exp_moment_F_Ftilde}. The first step, detailed in Section \ref{section:sum_martingale}, consists in giving a new expression to the martingale increments and to control them. The second step, which is outlined in Section \ref{subsection:cond_moments}, focuses on managing the conditional moments of these increments to get Proposition \ref{prop:exp_moments_martingale_inc}. The proof of Proposition \ref{prop:exp_moment_F_Ftilde} is finally achieved in Section \ref{subsection:proof_exp_moment}.\\

Throughout the paper, constants may change from line to line and may depend on $\sigma$ without being specified.

\section{Control of the martingale increments}\label{section:sum_martingale}

For the sake of clarity, we set $\Delta=1$ in the sequel, so that by \eqref{eq:def_FY_FtildeY} we have $t_k=k$. When $\Delta>0$ is arbitrary, the arguments are the same, it sufficies to apply a rescaling.\\

Through equation \eqref{eq:sde} and the fact that $b$ is Lipschitz continuous, for all $t\geqslant0$,~ $Y_{t}$ can be seen as a measurable functional of the time $t$, the initial condition $x$ and the Brownian motion $(W_s)_{s\in[0,t]}$. Denote by $\Phi:\R_+\times\R^d\times\mathcal{C}(\R_+,\R^d)\to\R^d$ this functional, we then have 
\begin{equation}
\forall t\geqslant0,\quad Y_{t}:=\Phi_t(x,(W_s)_{s\in[0,t]}).
\end{equation}

Now, let $k\geqslant1$, we have
\begin{align}\label{eq:majo_M}
&|M_k-M_{k-1}|\nonumber\\
&=|\E[F_Y|\F_k]-\E[F_Y|\F_{k-1}]|\nonumber\\
&\leqslant\int_{\Omega}\left|F\left(\Phi_1\left(x,W_{[0,1]}\right),\dots,\Phi_k\left(x,W_{[0,k]}\right),\Phi_{k+1}\left(x,W_{[0,k]}\sqcup\tilde{w}_{[k,k+1]}\right),\dots,\Phi_n\left(x,W_{[0,k]}\sqcup\tilde{w}_{[k,n]}\right)\right)\right.\nonumber\\
&\left.~~~-F\left(\Phi_1\left(x,W_{[0,1]}\right),\dots,\Phi_{k-1}\left(x,W_{[0,k-1]}\right),\Phi_{k}\left(x,W_{[0,k-1]}\sqcup\tilde{w}_{[k-1,k]}\right),\dots,\Phi_n\left(x,W_{[0,k-1]}\sqcup\tilde{w}_{[k-1,n]}\right)\right)\right|\PP_W({\rm d}\tilde{w})\nonumber\\
&\leqslant\|F\|_{\rm Lip}\int_{\Omega}\sum_{t=k}^n\left|\Phi_t\left(x,W_{[0,k]}\sqcup\tilde{w}_{[k,t]}\right)-\Phi_t\left(x,W_{[0,k-1]}\sqcup\tilde{w}_{[k-1,t]}\right)\right|\PP_W({\rm d}\tilde{w}).
\end{align}

With exactly the same procedure, we get

\begin{align}\label{eq:majo_Mtilde}
|\tilde{M}_k-\tilde{M}_{k-1}|\leqslant\|\tilde{F}\|_{\rm Lip}\int_{\Omega}\int_{k-1}^T\left|\Phi_t\left(x,W_{[0,k]}\sqcup\tilde{w}_{[k,t]}\right)-\Phi_t\left(x,W_{[0,k-1]}\sqcup\tilde{w}_{[k-1,t]}\right)\right|{\rm d}t~\PP_W({\rm d}\tilde{w}).
\end{align}

Let us introduce now some notations. First, for all $t\geqslant0$ set $u:=t-k+1$, then for all $u\geqslant0$, we define 
$$
X_u:=\left\{\begin{array}{lll}
\Phi_{u+k-1}\left(x,(W_s)_{s\in[0,k]}\sqcup(\tilde{w}_s)_{s\in[k,u+k-1]}\right) & \text{if } u\geqslant1 \\
\Phi_{u+k-1}\left(x,(W_s)_{s\in[0,u+k-1]}\right) & \text{otherwise},  
\end{array}\right.
$$
and
$$
\tilde{X}_u:=\Phi_{u+k-1}\left(x,(W_s)_{s\in[0,k-1]}\sqcup(\tilde{w}_s)_{s\in[k-1,u+k-1]}\right).
$$
We then have
\begin{align}\label{eq:sde_X}
X_{u}=&X_0+\int_{0}^u b(X_s){\rm d}s +\sigma\int_{0}^{k-1}(K^{}_H(u+k-1,s)-K^{}_H(k-1,s)){\rm d}W_s\nonumber\\
&\quad+\sigma\int_{k-1}^{k\wedge(u+k-1)}K^{}_H(u+k-1,s){\rm d}W_s+\sigma\int_{k}^{k\vee(u+k-1)}K^{}_H(u+k-1,s){\rm d}\tilde{w}_s
\end{align}
and
\begin{align}\label{eq:sde_Xtilde}
\tilde{X}_{u}=&\tilde{X}_0+\int_{0}^u b(\tilde{X}_s){\rm d}s +\sigma\int_{0}^{k-1}(K^{}_H(u+k-1,s)-K^{}_H(k-1,s)){\rm d}W_s+\sigma\int_{k-1}^{u+k-1}K^{}_H(u+k-1,s){\rm d}\tilde{w}_s.
\end{align}
\begin{rem} Let us note that the integrals involving $\tilde{w}$ in \eqref{eq:sde_X} and \eqref{eq:sde_Xtilde} and in the sequel have to be seen as Wiener integrals, so that they are defined $\PP_W({\rm d}\tilde{w})$ almost surely.
\end{rem}
Since $X_{0}=\tilde{X}_{0}=\Phi_{k-1}\left(x,(W_s)_{s\in[0,k-1]}\right)$, we deduce from $\eqref{eq:sde_X}$ and $\eqref{eq:sde_Xtilde}$ that for all $u\geqslant0$
\begin{align}\label{eq:X-Xtilde}
X_{u}-\tilde{X}_{u}&=\int_{0}^u b(X_s)-b(\tilde{X}_s){\rm d}s+\sigma\int_{k-1}^{k\wedge(u+k-1)}K^{}_H(u+k-1,s){\rm d}(W-\tilde{w})_s\nonumber\\
&=\int_{0}^u b(X_s)-b(\tilde{X}_s){\rm d}s+\sigma\int_{0}^{1\wedge u}K^{}_H(u+k-1,s+k-1){\rm d}(W^{(k)}-\tilde{w}^{(k)})_s
\end{align}
where we have set $(W^{(k)}_s)_{s\geqslant 0}:=(W_{s+k-1}-W_{k-1})_{s\geqslant0}$ which is a Brownian motion independent from $\F_{k-1}$ and $(\tilde{w}^{(k)}_s)_{s\geqslant0}:=(\tilde{w}_{s+k-1}-\tilde{w}_{k-1})_{s\geqslant0}$.\\

In the remainder of the section, we proceed to a control of the quantity $|X_u-\tilde{X}_u|$.
We have the following upper bound on $|X_u-\tilde{X}_u|$:

\begin{propo}\label{prop:control_X_Xtilde}
There exists $C_H>0$ such that for all $u>0$ and $k\in\N^*$,
\begin{align}
&|X_u-\tilde{X}_u|\nonumber\\
&\quad\leqslant C_H\sqrt{\Psi_H(u\vee1,k)}\left(\sup\limits_{v\in[0,1]}|W^{(k)}_v-\tilde{w}^{(k)}_v|+\sup\limits_{v\in[0,1/2]}\left|\int_0^1s^{\frac{1}{2}-H}\left(1-vs\right)^{H-\frac{3}{2}}{\rm d}(W^{(k)}-\tilde{w}^{(k)})_s\right|+\sup\limits_{v\in[0,2]}|G^{(k)}_v|\right)
\end{align}
where $X_u-\tilde{X}_u$ is defined in \eqref{eq:X-Xtilde}, $\Psi_H$ is defined by
$$\Psi_H(u,k):=C'_H\left\{\begin{array}{lll}u^{2H-3}& \text{if} &H\in(0,1/2)\\
k^{1-2H}u^{4H-4}+u^{2H-3}& \text{if} &H\in(1/2,1)\end{array}\right.$$ with $C'_H>0$ and $G^{(k)}$ is given by 
\begin{equation*}
G_v^{(k)}=\int_0^{1\wedge v}K^{}_H(v+k-1,s+k-1){\rm d}(W^{(k)}-\tilde{w}^{(k)})_s.
\end{equation*}
\end{propo}

In Subsections \ref{subsection:first_case} and \ref{subsection:second_case}, we prove Proposition \ref{prop:control_X_Xtilde}.

\subsection{First case : $u\geqslant2$}\label{subsection:first_case}
\subsubsection{When $k\neq1$}
\begin{lem}\label{lem:control_X_Xtilde_first_case} Let $k\neq1$. Then, for all $u\geqslant2$,
\begin{equation*}
|X_{u}-\tilde{X}_{u}|^2\leqslant e^{-\alpha(u-2)}|X_2-\tilde{X}_2|^2+ \Psi_H(u,k)\sup\limits_{s\in[0,1]}|W_s^{(k)}-\tilde{w}^{(k)}_s|^2
\end{equation*}
where $\Psi_H$ is defined in Proposition \ref{prop:control_X_Xtilde}.
\end{lem}

\begin{proof}
Let $u\geqslant2$.
In the following inequalities, we make use of Hypothesis \ref{hyp:b-coercive} on the function $b$ and of the elementary Young inequality $\langle a,b\rangle\leqslant\frac{1}{2}\left(\varepsilon|a|^2+\frac{1}{\varepsilon}|b|^2\right)$ with $\varepsilon=2\alpha$. By \eqref{eq:X-Xtilde},
\begin{align}
\frac{{\rm d}}{{\rm d}u}|X_{u}-\tilde{X}_{u}|^2&=2\langle X_u-\tilde{X}_u,b(X_u)-b(\tilde{X}_u)\rangle+\langle X_u-\tilde{X}_u,~\sigma\int_0^1\frac{\partial}{\partial u}K^{}_H(u+k-1,s+k-1){\rm d}(W^{(k)}-\tilde{w}^{(k)})_s\rangle\nonumber\\
&\leqslant -2\alpha|X_{u}-\tilde{X}_{u}|^2+\alpha|X_{u}-\tilde{X}_{u}|^2+\frac{|\sigma|^2}{2\alpha}\left|\int_0^1\frac{\partial}{\partial u}K^{}_H(u+k-1,s+k-1){\rm d}(W^{(k)}-\tilde{w}^{(k)})_s\right|^2\nonumber\\
&\leqslant -\alpha|X_{u}-\tilde{X}_{u}|^2+\frac{|\sigma|^2}{2\alpha}\left|\int_0^1\frac{\partial}{\partial u}K^{}_H(u+k-1,s+k-1){\rm d}(W^{(k)}-\tilde{w}^{(k)})_s\right|^2.\nonumber
\end{align}
We then apply Gronwall's lemma to obtain
\begin{equation}\label{eq:proof2}
|X_{u}-\tilde{X}_{u}|^2\leqslant e^{-\alpha(u-2)}|X_2-\tilde{X}_2|^2+\frac{|\sigma|^2}{2\alpha}\int_{2}^ue^{-\alpha(u-v)}\left|\int_0^1\frac{\partial}{\partial v}K^{}_H(v+k-1,s+k-1){\rm d}(W^{(k)}-\tilde{w}^{(k)})_s\right|^2{\rm d}v.
\end{equation} 
Now, we set for all $v\geqslant2$,
\begin{equation}\label{def:phi_k}
\varphi_k(v):=\int_0^1\frac{\partial}{\partial v}K^{}_H(v+k-1,s+k-1){\rm d}(W^{(k)}-\tilde{w}^{(k)})_s=c_H\int_0^1\left(\frac{v+k-1}{s+k-1}\right)^{H-\frac{1}{2}}(v-s)^{H-\frac{3}{2}}{\rm d}(W^{(k)}-\tilde{w}^{(k)})_s.
\end{equation} 

We apply an integration by parts to $\varphi_k$ taking into account that $W^{(k)}_0=\tilde{w}^{(k)}_0=0$:
\begin{align}\label{eq:phi_k_IPP}
\varphi_k(v)&=c_H\left(\frac{v+k-1}{k}\right)^{H-\frac{1}{2}}(v-1)^{H-\frac{3}{2}}(W^{(k)}_1-\tilde{w}^{(k)}_1)\nonumber\\
&\quad\quad -c_H(1/2-H)\int_0^1(v+k-1)^{H-\frac{1}{2}}(s+k-1)^{-H-\frac{1}{2}}(v-s)^{H-\frac{3}{2}}(W^{(k)}_s-\tilde{w}^{(k)}_s){\rm d}s\nonumber\\
&\quad\quad\quad -c_H(3/2-H)\int_0^1\left(\frac{v+k-1}{s+k-1}\right)^{H-\frac{1}{2}}(v-s)^{H-\frac{5}{2}}(W^{(k)}_s-\tilde{w}^{(k)}_s){\rm d}s\nonumber\\
&=:c_H(I_1(v)+I_2(v)+I_3(v)).
\end{align}
Recall that by \eqref{eq:proof2}, our goal here is to manage 
\begin{align}\label{eq:term_to_control}
&\int_2^ue^{-\alpha(u-v)}|\varphi_k(v)|^2{\rm d}v\nonumber\\
&\quad\quad\leqslant 3c_H^2\left(\int_2^ue^{-\alpha(u-v)}|I_1(v)|^2{\rm d}v+\int_2^ue^{-\alpha(u-v)}|I_2(v)|^2{\rm d}v+\int_2^ue^{-\alpha(u-v)}|I_3(v)|^2{\rm d}v\right)
\end{align} 

To control each term involving $I_1$, $I_2$ and $I_3$ in \eqref{eq:term_to_control}, we will need the following inequality:

\begin{align}\label{eq:terme_commun}
&\int_2^ue^{-\alpha(u-v)}k^{1-2H}(v-1+k)^{2H-1}(v-1)^{2H-3}{\rm d}v\nonumber\\
&\quad\quad\quad\leqslant C_H\left\{\begin{array}{lll}
k^{1-2H}(u-1)^{4H-4}+(u-1)^{2H-3} & \text{for} & H>1/2 \\
(u-1)^{2H-3} & \text{for} & H<1/2 
\end{array}\right..
\end{align}
Inequality \eqref{eq:terme_commun} is obtained through Lemma \ref{lem:control_integral} and the elementary inequalities $(v-1+k)^{2H-1}\leqslant (v-1)^{2H-1}+k^{2H-1}$ if $H>1/2$ and $(v-1+k)^{2H-1}\leqslant k^{2H-1}$ if $H<1/2$.

\begin{lem}\label{lem:control_integral} Let $\alpha,\beta>0$. Then, for all $u\geqslant2$,
\begin{equation*}
\int_{2}^u e^{-\alpha (u-v)} (v-1)^{-\beta}{\rm d}v\leqslant C_{\alpha,\beta}(u-1)^{-\beta}.
\end{equation*}
\end{lem}
\begin{proof}
It is enough to apply an integration by parts and then use that $$\sup\limits_{v\in[2,u]}e^{-\alpha(u-v)}(v-1)^{-\beta-1}=\max\left(e^{-\alpha(u-2)},(u-1)^{-\beta-1}\right)$$ to conclude the proof.
\end{proof}
It remains to show how the terms involving $I_1$, $I_2$ and $I_3$ in \eqref{eq:term_to_control} can be reduced to the term \eqref{eq:terme_commun}. Let us begin with $I_1$ which is straightforward:
\begin{align}\label{eq:I_1_control}
\int_2^ue^{-\alpha(u-v)}|I_1(v)|^2{\rm d}v&\leqslant |W^{(k)}_1-w^{(k)}_1|^2\int_2^ue^{-\alpha(u-v)}k^{1-2H}(v-1+k)^{2H-1}(v-1)^{2H-3}{\rm d}v\nonumber\\
&\leqslant \sup\limits_{s\in[0,1]}|W^{(k)}_s-w^{(k)}_s|^2 \int_2^ue^{-\alpha(u-v)}k^{1-2H}(v-1+k)^{2H-1}(v-1)^{2H-3}{\rm d}v .
\end{align}
Then, using the definition of $I_2$, 
\begin{align}\label{eq:I_2_control}
&\int_2^ue^{-\alpha(u-v)}|I_2(v)|^2{\rm d}v\nonumber\\
&\quad\quad\leqslant(1/2-H)^2 \int_2^ue^{-\alpha(u-v)}(v-1+k)^{2H-1}(v-1)^{2H-3}(k-1)^{-2H-1}\left(\int_{0}^1|W^{(k)}_s-w^{(k)}_s|{\rm d}s\right)^2{\rm d}v\nonumber\\
&\quad\quad\leqslant C_H \sup\limits_{s\in[0,1]}|W^{(k)}_s-w^{(k)}_s|^2 \int_2^ue^{-\alpha(u-v)}k^{1-2H}(v-1+k)^{2H-1}(v-1)^{2H-3}{\rm d}v.
\end{align}
Finally,
\begin{align}\label{eq:I_3_control}
&\int_2^ue^{-\alpha(u-v)}|I_3(v)|^2{\rm d}v\nonumber\\
&\quad\quad\leqslant(3/2-H)^2 \int_2^ue^{-\alpha(u-v)}(v-1+k)^{2H-1}(v-1)^{2H-5}\left(\int_{0}^1(s+k-1)^{\frac{1}{2}-H}|W^{(k)}_s-w^{(k)}_s|{\rm d}s\right)^2{\rm d}v\nonumber\\
&\quad\quad\leqslant C_H \sup\limits_{s\in[0,1]}|W^{(k)}_s-w^{(k)}_s|^2 \int_2^ue^{-\alpha(u-v)}(v-1+k)^{2H-1}(v-1)^{2H-3}\left(\int_{0}^1(s+k-1)^{\frac{1}{2}-H}{\rm d}s\right)^2{\rm d}v\nonumber\\
&\quad\quad\leqslant C'_H \sup\limits_{s\in[0,1]}|W^{(k)}_s-w^{(k)}_s|^2 \int_2^ue^{-\alpha(u-v)}k^{1-2H}(v-1+k)^{2H-1}(v-1)^{2H-3}{\rm d}v
\end{align}
where the last inequality is given by the following fact: there exists $C_H>0$ such that for all $k\neq1$, $\sup\limits_{s\in[0,1]}(s+k-1)^{\frac{1}{2}-H}\leqslant C_Hk^{\frac{1}{2}-H}$.\\
It remains to combine the three above inequalities \eqref{eq:I_1_control}, \eqref{eq:I_2_control} and \eqref{eq:I_3_control} with \eqref{eq:terme_commun}  to get the following in \eqref{eq:term_to_control}:
\begin{equation*}
\int_2^ue^{-\alpha(u-v)}|\varphi_k(v)|^2{\rm d}v\leqslant C_H\sup\limits_{s\in[0,1]}|W^{(k)}_s-w^{(k)}_s|^2\left\{\begin{array}{lll}
k^{1-2H}(u-1)^{4H-4}+(u-1)^{2H-3} & \text{for} & H>1/2 \\
(u-1)^{2H-3} & \text{for} & H<1/2 
\end{array}\right..
\end{equation*} 
Putting this inequality into \eqref{eq:proof2} gives the result (we can replace $u-1$ by $u$, the inequality remains true when $u\geqslant2$ up to a constant).
\end{proof}

\subsubsection{When $k=1$}

\begin{lem}\label{lem:control_X_Xtilde_second_case} Let $k=1$. Then, for all $u\geqslant2$,
\begin{equation*}
|X_{u}-\tilde{X}_{u}|^2\leqslant e^{-\alpha(u-2)}|X_2-\tilde{X}_2|^2+ \Psi_H(u,1)\sup\limits_{v\in[0,1/2]}\left|\int_0^1 s^{\frac{1}{2}-H}(1-vs)^{H-\frac{3}{2}}{\rm d}\left(W^{(1)}-\tilde{w}^{(1)}\right)_s\right|^2
\end{equation*}
where $\Psi_H$ is defined in Proposition \ref{prop:control_X_Xtilde}.
\end{lem}

\begin{proof}
The proof begins as in the proof of Lemma \ref{lem:control_X_Xtilde_first_case}. We have through inequality \eqref{eq:proof2}:
\begin{equation}\label{eq:proof_case_2}
|X_{u}-\tilde{X}_{u}|^2\leqslant e^{-\alpha(u-2)}|X_2-\tilde{X}_2|^2+\frac{|\sigma|^2}{2\alpha}\int_{2}^ue^{-\alpha(u-v)}\left|\varphi_1(v)\right|^2{\rm d}v
\end{equation}
with 
\begin{align}
\varphi_1(v)&=c_Hv^{H-\frac{1}{2}}\int_0^1 s^{\frac{1}{2}-H}(v-s)^{H-\frac{3}{2}}{\rm d}(W^{(1)}-\tilde{w}^{(1)})_s\nonumber\\
&=c_Hv^{H-\frac{1}{2}}v^{H-\frac{3}{2}}\int_0^1 s^{\frac{1}{2}-H}\left(1-\frac{s}{v}\right)^{H-\frac{3}{2}}{\rm d}(W^{(1)}-\tilde{w}^{(1)})_s.
\end{align}
Since for $v\geqslant2$, $v^{H-\frac{1}{2}}$ is bounded when $H<1/2$, we have
\begin{align*}
\int_{2}^ue^{-\alpha(u-v)}\left|\varphi_1(v)\right|^2{\rm d}v&\leqslant c_H \int_{2}^ue^{-\alpha(u-v)} v^{(4H-4)\vee(2H-3)}\left|\int_0^1 s^{\frac{1}{2}-H}\left(1-\frac{s}{v}\right)^{H-\frac{3}{2}}{\rm d}(W^{(1)}-\tilde{w}^{(1)})_s\right|^2{\rm d}v\\
&\leqslant C_H\sup\limits_{v'\in[0,1/2]}\left|\int_0^1 s^{\frac{1}{2}-H}\left(1-v's\right)^{H-\frac{3}{2}}{\rm d}(W^{(1)}-\tilde{w}^{(1)})_s\right|^2\int_{2}^ue^{-\alpha(u-v)} v^{(4H-4)\vee(2H-3)}{\rm d}v.
\end{align*}
Then, we use Lemma \ref{lem:control_integral} in the previous inequality, which gives:
\begin{equation*}
\int_{2}^ue^{-\alpha(u-v)}\left|\varphi_1(v)\right|^2{\rm d}v\leqslant C_H\sup\limits_{v'\in[0,1/2]}\left|\int_0^1 s^{\frac{1}{2}-H}\left(1-v's\right)^{H-\frac{3}{2}}{\rm d}(W^{(1)}-\tilde{w}^{(1)})_s\right|^2(u-1)^{(4H-4)\vee(2H-3)}.
\end{equation*}
This inequality combined with \eqref{eq:proof_case_2} concludes the proof (we can replace $u-1$ by $u$, the inequality remains true when $u\geqslant2$ up to a constant).
\end{proof}

\subsection{Second case : $u\in[0,2]$}\label{subsection:second_case}
The idea here is to use Gronwall lemma in its integral form. By Hypothesis \ref{hyp:b-coercive}, $b$ is $L$-Lipschitz so that:
\begin{equation*}
|X_u-\tilde{X}_u|\leqslant L\int_{0}^u|X_s-\tilde{X}_s|{\rm d}s+\left|\int_0^{1\wedge u}K^{}_H(u+k-1,s+k-1){\rm d}(W^{(k)}-\tilde{w}^{(k)})_s\right|.
\end{equation*}
Then, for $u\in[0,2]$,
\begin{align}\label{eq:third_case}
|X_u-\tilde{X}_u|&\leqslant\left|\int_0^{1\wedge u}K^{}_H(u+k-1,s+k-1){\rm d}(W^{(k)}-\tilde{w}^{(k)})_s\right|\nonumber\\
&\quad\quad\quad+\int_0^u\left|\int_0^{1\wedge v}K^{}_H(v+k-1,s+k-1){\rm d}(W^{(k)}-\tilde{w}^{(k)})_s\right|e^{L(u-v)}{\rm d}v\nonumber\\
&\leqslant e^{2L}\sup\limits_{v\in[0,2]}\left|\int_0^{1\wedge v}K^{}_H(v+k-1,s+k-1){\rm d}(W^{(k)}-\tilde{w}^{(k)})_s\right|
\end{align}
For all $k\geqslant1$ and for all $v\in[0,2]$, we set 
\begin{equation}\label{def:G}
G_v^{(k)}(W-\tilde{w})=\int_0^{1\wedge v}K^{}_H(v+k-1,s+k-1){\rm d}(W^{(k)}-\tilde{w}^{(k)})_s.
\end{equation}
\bigskip

The inequality \eqref{eq:third_case} combined with Lemma \ref{lem:control_X_Xtilde_first_case} and Lemma \ref{lem:control_X_Xtilde_second_case} finally prove Proposition \ref{prop:control_X_Xtilde}.

\section{Conditional exponential moments of the martingale increments}\label{section:moments}
\subsection{Conditional moments of the martingale increments}\label{subsection:cond_moments}

\begin{propo}\label{prop:moments_martingale_inc}
\begin{itemize}
\item[(i)] There exists $C,\zeta>0$ such that for all $k\in\N^*$ and for all $p\geqslant2$,
\begin{equation}
\E[|M_k-M_{k-1}|^p|\mathcal{F}_{k-1}]^{1/p}\leqslant C\|F\|_{\rm Lip}\psi_{n,k}\left(\zeta^{p/2}p\Gamma\left(\frac{p}{2}\right)\right)^{1/p}\quad a.s.
\end{equation}
\item[(ii)]There exists $C,\zeta>0$ such that for all $k\in\N^*$ and for all $p\geqslant2$,
\begin{equation}
\E[|\tilde{M}_k-\tilde{M}_{k-1}|^p|\mathcal{F}_{k-1}]^{1/p}\leqslant C\|\tilde{F}\|_{\rm Lip}\psi'_{T,k}\left(\zeta^{p/2}p\Gamma\left(\frac{p}{2}\right)\right)^{1/p}\quad a.s.
\end{equation}
\end{itemize}
where $\psi_{n,k}:=\sum_{u=1}^{n-k+1}\sqrt{\Psi_H(u,k)}$, $~\psi'_{T,k}:=\int_{0}^{T-k+1}\sqrt{\Psi_H(u\vee1,k)}{\rm d}u~$ and $\Psi_H$ is defined in Proposition~\ref{prop:control_X_Xtilde}.

\end{propo}

To prove this result, we first need the following intermediate outcome.

\begin{lem}\label{lem:moments_martingale_inc}
For all $k\in\N^*$, let $G^{(k)}$ be defined by \eqref{def:G}. Then, for all $p\geqslant2$, there exists $C>0$ such that
\begin{align*}
&\E[|M_k-M_{k-1}|^p|\mathcal{F}_{k-1}]^{1/p}\\
&\leqslant3C\|F\|_{\rm Lip}\psi_{n,k}\left(\E\left[\sup\limits_{v\in[0,1]}|W^{(1)}_v-\tilde{W}^{(1)}_v|^p\right]^{1/p}
+\E\left[\sup\limits_{v\in[0,1/2]}\left|\int_0^1s^{\frac{1}{2}-H}\left(1-vs\right)^{H-\frac{3}{2}}{\rm d}(W^{(1)}-\tilde{W}^{(1)})_s\right|^p\right]^{1/p}\right.\\
&\quad\quad\quad\quad\quad\quad\quad\quad\quad\quad\quad\quad\left.+\E\left[\sup\limits_{v\in[0,2]}|G^{(k)}_v(W-\tilde{W})|^p\right]^{1/p}\right)\quad a.s.
\end{align*}
Or equivalently, since $W^{(k)}$ and $\tilde{W}^{(k)}$ are iid we can replace $W^{(k)}-\tilde{W}^{(k)}$ by $\sqrt{2}W^{(k)}$:
\begin{align*}
&\E[|M_k-M_{k-1}|^p|\mathcal{F}_{k-1}]^{1/p}\\
&\leqslant3C\sqrt{2}\|F\|_{\rm Lip}\psi_{n,k}\left(\E\left[\sup\limits_{v\in[0,1]}|W^{(1)}_v|^p\right]^{1/p}+\E\left[\sup\limits_{v\in[0,1/2]}\left|\int_0^1s^{\frac{1}{2}-H}\left(1-vs\right)^{H-\frac{3}{2}}{\rm d}W^{(1)}_s\right|^p\right]^{1/p}\right.\\
&\quad\quad\quad\quad\quad\quad\quad\quad\quad\quad\quad\quad\left.+\E\left[\sup\limits_{v\in[0,2]}|G^{(k)}_v(W)|^p\right]^{1/p}\right)\quad a.s.
\end{align*}
where $\psi_{n,k}=\sum_{u=1}^{n-k+1}\sqrt{\Psi_H(u,k)}$ and $\Psi_H$ is defined in Proposition \ref{prop:control_X_Xtilde}.\\
The same occurs for $\tilde{M}$ instead of $M$ by replacing $F$ by $\tilde{F}$ and $\psi_{n,k}$ by $\psi'_{T,k}=\int_{0}^{T-k+1}\sqrt{\Psi_H(u\vee1,k)}{\rm d}u$.
\end{lem}

\begin{proof}For the sake of simplicity, assume that $\|F\|_{\rm Lip}=1$. By inequality \eqref{eq:majo_M}, we have for all $p\geqslant2$,
\begin{equation*}
|M_k-M_{k-1}|^p\leqslant\left(\int_\Omega\sum_{u=1}^{n-k+1}|X_u-\tilde{X}_u|~\PP_W(d\tilde{w})\right)^p.
\end{equation*}
Now, we use Proposition \ref{prop:control_X_Xtilde} and for the sake of clarity we set $\|W^{(k)}-\tilde{w}^{(k)}\|_{\infty, [0,1]}:=\sup\limits_{v\in[0,1]}|W^{(k)}_v-\tilde{w}^{(k)}_v|$, $A(W^{(k)}-\tilde{w}^{(k)}):=\sup\limits_{v\in[0,1/2]}\left|\int_0^1s^{\frac{1}{2}-H}\left(1-vs\right)^{H-\frac{3}{2}}{\rm d}(W^{(k)}-\tilde{w}^{(k)})_s\right|$ and \\$C_k(W^{(k)}-\tilde{w}^{(k)}):=\sup\limits_{v\in[0,2]}|G^{(k)}_v(W-\tilde{w})|$. Then, by Jensen inequality,
\begin{align*}
&|M_k-M_{k-1}|^p\\
&\leqslant C^p\psi_{n,k}^p\left(\int_\Omega \|W^{(k)}-\tilde{w}^{(k)}\|_{\infty, [0,1]}+A(W^{(k)}-\tilde{w}^{(k)})+C_k(W^{(k)}-\tilde{w}^{(k)})~\PP_W(d\tilde{w})\right)^p\\
&\leqslant 3^{p-1}C^p\psi_{n,k}^p\left[\int_\Omega \|W^{(k)}-\tilde{w}^{(k)}\|_{\infty, [0,1]}^p~\PP_W(d\tilde{w})+\int_\Omega \left(A(W^{(k)}-\tilde{w}^{(k)})\right)^p~\PP_W(d\tilde{w})\right.\\
&\quad\quad\quad\quad\quad\quad\quad\quad\left.+\int_\Omega \left(C_k(W^{(k)}-\tilde{w}^{(k)})\right)^p~\PP_W(d\tilde{w})\right]
\end{align*}
Recall that $W^{(k)}=(W_{s+k-1}-W_{k-1})_{s\geq0}$ and thus $W^{(k)}$ is independent of $\mathcal{F}_{k-1}$.
Then,
\begin{align*}
&\E[|M_k-M_{k-1}|^p|\mathcal{F}_{k-1}]\\
&\leq 3^{p-1}C^p\psi_{n,k}^p\E\left[\int_\Omega \|W^{(k)}-\tilde{w}^{(k)}\|_{\infty, [0,1]}^p~\PP_W(d\tilde{w})+\int_\Omega \left(A(W^{(k)}-\tilde{w}^{(k)})\right)^p~\PP_W(d\tilde{w})\right.\\
&\quad\quad\quad\quad\quad\quad\quad\quad\left.+\left.\int_\Omega \left(C_k(W^{(k)}-\tilde{w}^{(k)})\right)^p~\PP_W(d\tilde{w})\right|\mathcal{F}_{k-1}\right]\\
&\leq 3^{p-1}C^p\psi_{n,k}^p\E\left[\int_\Omega \|W^{(k)}-\tilde{w}^{(k)}\|_{\infty, [0,1]}^p~\PP_W(d\tilde{w})+\int_\Omega \left(A(W^{(k)}-\tilde{w}^{(k)})\right)^p~\PP_W(d\tilde{w})\right.\\
&\quad\quad\quad\quad\quad\quad\quad\quad\left.+\int_\Omega \left(C_k(W^{(k)}-\tilde{w}^{(k)})\right)^p~\PP_W(d\tilde{w})\right]
\end{align*}
We denote by $\mathcal{F}^{(k)}$ the filtration associated to $W^{(k)}$, we rewrite
\begin{align*}
&\E\left[|M_k-M_{k-1}|^p|\mathcal{F}_{k-1}\right]\\
&\quad\leqslant 3^{p-1}C^p\psi_{n,k}^p\left(\E\left[\E\left[\left.\|W^{(k)}-\tilde{W}^{(k)}\|_{\infty, [0,1]}^p\right|\mathcal{F}_1^{(k)}\right]\right]+\E\left[\E\left[\left.\left(A(W^{(k)}-\tilde{W}^{(k)})\right)^p\right|\mathcal{F}_{1}^{(k)}\right]\right]\right.\\
&\quad\quad\quad\quad\quad\quad\quad\quad\quad\quad\quad\quad\left.+\E\left[\E\left[\left.\left(C_k(W^{(k)}-\tilde{W}^{(k)})\right)^p\right|\mathcal{F}_1^{(k)}\right]\right]\right)\\
&\quad\quad=3^{p-1}C^p\psi_{n,k}^p\left(\E\left[\|W^{(k)}-\tilde{W}^{(k)}\|_{\infty, [0,1]}^p\right]+\E\left[\left(A(W^{(k)}-\tilde{W}^{(k)})\right)^p\right]+\E\left[\left(C_k(W^{(k)}-\tilde{W}^{(k)})\right)^p\right]\right)\\
\end{align*}
Using the elementary inequality $(a+b)^{1/p}\leqslant a^{1/p}+b^{1/p}$, we finally get :
\begin{align*}
&\E\left[|M_k-M_{k-1}|^p|\mathcal{F}_{k-1}\right]^{1/p}\\
&\leqslant 3C\psi_{n,k}\left(\E\left[\|W^{(k)}-\tilde{W}^{(k)}\|_{\infty, [0,1]}^p\right]^{1/p}+\E\left[\left(A(W^{(1)}-\tilde{W}^{(1)})\right)^p\right]^{1/p}+\E\left[\left(C_k(W^{(k)}-\tilde{W}^{(k)})\right)^p\right]^{1/p}\right)
\end{align*}
and the proof is over since $W^{(k)}$ and $\tilde{W}^{(k)}$ have respectively the same distribution as $W^{(1)}$ and $\tilde{W}^{(1)}$.\\
In the same way, we prove the result for $\tilde{M}$ by using \eqref{eq:majo_Mtilde} which gives
\begin{equation*}
|\tilde{M}_k-\tilde{M}_{k-1}|^p\leqslant\left(\int_\Omega\int_{0}^{T-k+1}|X_u-\tilde{X}_u|~\PP_W(d\tilde{w})\right)^p
\end{equation*}
and Proposition \ref{prop:control_X_Xtilde}.
\end{proof}

\begin{proof}[Proof of Proposition \ref{prop:moments_martingale_inc}]
With Lemma \ref{lem:moments_martingale_inc} in hand, we just need to prove that there exist $\zeta>0$ such that for all $k\in\N^*$ and for all $p\geqslant2$
\begin{align}
\E\left[\sup\limits_{v\in[0,1]}|W^{(1)}_v|^p\right]^{1/p}&\leqslant \left(\zeta^{p/2}p\Gamma\left(\frac{p}{2}\right)\right)^{1/p},\label{eq:proof_sub_gaussian1}\\
\E\left[\sup\limits_{v\in[0,1/2]}\left|\int_0^1s^{\frac{1}{2}-H}\left(1-vs\right)^{H-\frac{3}{2}}{\rm d}W^{(1)}_s\right|^p\right]^{1/p}&\leqslant \left(\zeta^{p/2}p\Gamma\left(\frac{p}{2}\right)\right)^{1/p}\label{eq:proof_sub_gaussian2}\\
\text{and }\quad\E\left[\sup\limits_{v\in[0,2]}|G^{(k)}_v(W)|^p\right]^{1/p}&\leqslant\left(\zeta^{p/2}p\Gamma\left(\frac{p}{2}\right)\right)^{1/p}.\label{eq:proof_sub_gaussian3}
\end{align}
Condition \eqref{eq:proof_sub_gaussian1} is given in Appendix \ref{section:sub_gaussian_BM} and condition \eqref{eq:proof_sub_gaussian3} follows from Proposition \ref{prop:uniform_sub_gaussian_G} since $$\E\left[\sup\limits_{v\in[0,2]}|G^{(k)}_v(W)|^p\right]^{1/p}\leqslant2^{\alpha'_H}\E\left[\|G^{(k)}\|_{\alpha'_H,[0,2]}^p\right]^{1/p}$$
where $\alpha'_H\in(0,1)$ is defined in Proposition \ref{prop:uniform_sub_gaussian_G}.
Hence, it remains to get \eqref{eq:proof_sub_gaussian2}. To this end, we set for all $v\in[0,1/2]$, 
$$\tilde{G}_v:=\int_0^1s^{\frac{1}{2}-H}(1-vs)^{H-\frac{3}{2}}{\rm d}W^{(1)}_s.$$
Let $0\leqslant v'<v\leqslant 1/2$, we have
\begin{align*}
\E[|\tilde{G}_v-\tilde{G}_{v'}|^2]&=\int_0^1s^{1-2H}[(1-vs)^{H-\frac{3}{2}}-(1-v's)^{H-\frac{3}{2}}]^2{\rm d}s\\
&=\frac{1}{(3/2-H)^2}\int_0^1s^{1-2H}\left(\int_{v'}^v(1-us)^{H-\frac{5}{2}}{\rm d}u\right)^2{\rm d}s.
\end{align*}
Since for all $u\in[0,1/2]$ and for all $s\in[0,1]$ we have $\frac{1}{2}\leqslant1-us\leqslant1$, we deduce that 
\begin{align*}
\E[|\tilde{G}_v-\tilde{G}_{v'}|^2]\leqslant C_H(v-v')^2\int_{0}^1s^{1-2H}{\rm d}s=\frac{C_H}{2-2H}(v-v')^2.
\end{align*}
Hence, for all $\alpha\in(0,1)$, 
\begin{align}\label{eq:G_tilde}
\sup\limits_{0\leqslant v'<v\leqslant\frac{1}{2}}\frac{\E[|\tilde{G}_v-\tilde{G}_{v'}|^2]^{1/2}}{|v-v'|^\alpha}<+\infty
\end{align}
Now, following carefully the proof of Proposition \ref{prop:uniform_sub_gaussian_G} in Appendix \ref{section:uniform_sub_gaussian_G}, one can show that  \eqref{eq:G_tilde} and the fact that $\tilde{G}$ is a Gaussian process implies \eqref{eq:proof_sub_gaussian2} since for all $\alpha\in(0,1)$
$$\E\left[\sup\limits_{v\in[0,1/2]}|\tilde{G}_v|^p\right]^{1/p}\leqslant 2^{-\alpha}\E\left[\|\tilde{G}\|_{\alpha,[0,1/2]}^p\right]^{1/p}+\E[|\tilde{G}_0|^p]^{1/p}.$$
\end{proof}

\subsection{Proof of Proposition \ref{prop:exp_moment_F_Ftilde}}\label{subsection:proof_exp_moment}

We have the following result:
\begin{propo}\label{prop:exp_moments_martingale_inc}
\begin{itemize}
\item[(i)] There exists $C',\zeta>0$ such that for all $k\in\N^*$ and for all $\lambda>0$,
\begin{equation}
\E[\exp(\lambda(M_k-M_{k-1}))|\mathcal{F}_{k-1}]\leqslant \exp\left(2\lambda^2\|F\|^2_{\rm Lip}\psi^2_{n,k}C'\zeta\right)\quad a.s.
\end{equation}
\item[(ii)]There exists $C',\zeta>0$ such that for all $k\in\N^*$ and for all $\lambda>0$,
\begin{equation}
\E[\exp(\lambda(\tilde{M}_k-\tilde{M}_{k-1}))|\mathcal{F}_{k-1}]\leqslant \exp\left(2\lambda^2\|\tilde{F}\|^2_{\rm Lip}\psi'^2_{T,k}C'\zeta\right)\quad a.s.
\end{equation}
\end{itemize}
where $\psi_{n,k}:=\sum_{u=1}^{n-k+1}\sqrt{\Psi_H(u,k)}$, $~\psi'_{T,k}:=\int_{0}^{T-k+1}\sqrt{\Psi_H(u\vee1,k)}{\rm d}u~$ and $\Psi_H$ is defined in Proposition~\ref{prop:control_X_Xtilde}.

\end{propo}
\begin{proof}
Let us prove $(i)$. From $\E[M_k-M_{k-1}|\mathcal{F}_{k-1}]=0$ and Proposition \ref{prop:moments_martingale_inc}, we immediately get the result by using Lemma \ref{lem:moments_to_exponential_moments}.
\end{proof}
Let us now conclude the proof of Proposition \ref{prop:exp_moment_F_Ftilde} $(i)$. By the decomposition \eqref{eq:decomposition_sum_martingale_inc} and Proposition \ref{prop:exp_moments_martingale_inc} $(i)$, we have the following recursive inequality :
$$\E\left[e^{\lambda M_n}\right]=\E\left[e^{\lambda M_{n-1}}\E\left[\left.e^{\lambda(M_n-M_{n-1})}\right|\mathcal{F}_{n-1}\right]\right]\leq \exp\left(2\lambda^2\|F\|^2_{\rm Lip}\psi^2_{n,n}C'\zeta\right)\E\left[e^{\lambda M_{n-1}}\right]$$
which gives 
\begin{equation}\label{eq:almost_last}
\E\left[e^{\lambda M_n}\right]\leq \exp\left(2\lambda^2\|F\|^2_{\rm Lip}C'\zeta\sum_{k=1}^{n}\psi^2_{n,k}\right).
\end{equation}
Equation \eqref{eq:almost_last} combined with Lemma \ref{lem:estimate} (see below) finally proves Proposition \ref{prop:exp_moment_F_Ftilde} $(i)$. The proof of item $(ii)$ is exactly the same. 
\begin{lem} \label{lem:estimate}
\begin{itemize}
\item[(i)] Let $n\in\N^*$ and $(\psi_{n,k})$ be defined as in Proposition \ref{prop:moments_martingale_inc}. There exists $C_H>0$ such that 
\begin{align*}
\sum_{k=1}^n\psi^2_{n,k}\leqslant C_H ~n^{2\left(H\vee\frac{1}{2}\right)}.
\end{align*}
\item[(ii)] Let $T\geqslant1$ and $(\psi'_{T,k})$ be defined as in Proposition \ref{prop:moments_martingale_inc}. There exists $C_H>0$ such that 
\begin{align*}
\sum_{k=1}^{\lceil T\rceil}\psi'^2_{T,k}\leqslant C_H ~T^{2\left(H\vee\frac{1}{2}\right)}.
\end{align*}
\end{itemize}
\end{lem}

\begin{proof}
$(i)$ Recall that $\psi_{n,k}=\sum_{u=1}^{n-k+1}\sqrt{\Psi_H(u,k)}$ with
$$ \Psi_H(u,k):=C_H\left\{\begin{array}{lll}u^{2H-3}& \text{if} &H\in(0,1/2)\\
k^{1-2H}u^{4H-4}+u^{2H-3}& \text{if} &H\in(1/2,1)\end{array}\right.$$
and $C_H>0$. \\
$\rhd$ First case: $H\in(0,1/2)$. We have
$$\sum_{u=1}^{n-k+1}u^{H-\frac{3}{2}}\leqslant\sum_{u=1}^{+\infty}u^{H-\frac{3}{2}}<+\infty.$$
Then, 
$$\sum_{k=1}^n\psi^2_{n,k}\leqslant C_H ~n$$
which concludes the proof for $H\in(0,1/2)$.\\

$\rhd$ Second case: $H\in(1/2,1)$. We have
$$\sum_{u=1}^{n-k+1}u^{H-\frac{3}{2}}\leqslant\int_{0}^{n-k+1}t^{H-\frac{3}{2}}{\rm d}t=\frac{1}{H-1/2}(n-k+1)^{H-\frac{1}{2}}$$
and 
$$\sum_{u=1}^{n-k+1}u^{2H-2}\leqslant\int_{0}^{n-k+1}t^{2H-2}{\rm d}t=\frac{1}{2H-1}(n-k+1)^{2H-1}.$$
Then, 
\begin{align*}
\sum_{k=1}^n\psi^2_{n,k}&\leqslant C_{1,H}\sum_{k=1}^n(n-k+1)^{2H-1}+C_{2,H}\sum_{k=1}^nk^{1-2H}(n-k+1)^{4H-2}\\
&\leqslant C_{1,H}~n^{2H}+C_{2,H}(n+1)^{2H}\frac{1}{n+1}\sum_{k=1}^{n+1}\left(\frac{k}{n+1}\right)^{1-2H}\left(1-\frac{k}{n+1}\right)^{4H-2}.
\end{align*}
Since
$$\frac{1}{n+1}\sum_{k=1}^{n+1}\left(\frac{k}{n+1}\right)^{1-2H}\left(1-\frac{k}{n+1}\right)^{4H-2}\underset{n\to+\infty}{\longrightarrow}\int_0^1x^{1-2H}(1-x)^{4H-2}{\rm d}x<+\infty$$
we finally get the result when $H\in(1/2,1)$.\\

$(ii)$ Recall that $\psi'_{T,k}=\int_{0}^{T-k+1}\sqrt{\Psi_H(u\vee1,k)}{\rm d}u$.\\
$\rhd$ First case: $H\in(0,1/2)$. We have
$$\int_0^{T-k+1}(u\vee1)^{H-\frac{3}{2}}{\rm d}u\leqslant 1+\int_{1}^{+\infty}u^{H-\frac{3}{2}}{\rm d}u<+\infty.$$
Then, 
$$\sum_{k=1}^{\lceil T\rceil}\psi'^2_{T,k}\leqslant C_H ~\lceil T\rceil\leqslant\tilde{C}_H ~T$$
which concludes the proof for $H\in(0,1/2)$.\\

$\rhd$ Second case: $H\in(1/2,1)$. We have
$$\int_{0}^{T-k+1}(u\vee1)^{H-\frac{3}{2}}{\rm d}u=1+\frac{1}{H-1/2}[(T-k+1)^{H-1/2}-1]\leqslant\frac{1}{H-1/2}(T-k+1)^{H-\frac{1}{2}}$$
and 
$$\int_{0}^{T-k+1}(u\vee1)^{2H-2}{\rm d}u=1+\frac{1}{2H-1}[(T-k+1)^{2H-1}-1]\leqslant\frac{1}{2H-1}(T-k+1)^{2H-1}.$$
Then, 
\begin{align*}
\sum_{k=1}^{\lceil T\rceil}\psi'^2_{T,k}&\leqslant C_{1,H}\sum_{k=1}^{\lceil T\rceil}(T-k+1)^{2H-1}+C_{2,H}\sum_{k=1}^{\lceil T\rceil}k^{1-2H}(T-k+1)^{4H-2}\\
&\leqslant C_{1,H}~\lceil T\rceil T^{2H-1}+C_{2,H}(\lceil T\rceil+1)^{2H}\frac{1}{\lceil T\rceil+1}\sum_{k=1}^{\lceil T\rceil+1}\left(\frac{k}{\lceil T\rceil+1}\right)^{1-2H}\left(1-\frac{k}{\lceil T\rceil+1}\right)^{4H-2}.
\end{align*}
Since
$$\frac{1}{\lceil T\rceil+1}\sum_{k=1}^{\lceil T\rceil+1}\left(\frac{k}{\lceil T\rceil+1}\right)^{1-2H}\left(1-\frac{k}{\lceil T\rceil+1}\right)^{4H-2}\underset{T\to+\infty}{\longrightarrow}\int_0^1x^{1-2H}(1-x)^{4H-2}{\rm d}x<+\infty$$
we finally get the result when $H\in(1/2,1)$.
\end{proof}

\appendix
\section{Sub-Gaussianity of the supremum of the Brownian motion}\label{section:sub_gaussian_BM}
\begin{propo}\label{prop:sub_gaussian_sup_BM}
Let $(W_t)_{t\geqslant0}$ be a $d$-dimensional standard Brownian motion. There exist $\eta,\eta'>0$ such that
\begin{equation}\label{prop:condition1_sub_gaussian_sup_BM}
\forall x\geqslant0,\quad \PP\left(\sup\limits_{t\in[0,1]}|W_t|>x\right)\leqslant \eta'e^{-\eta x^2}
\end{equation}
Consequently, for all $p\geqslant2$,
\begin{equation}\label{prop:condition2_sub_gaussian_sup_BM}
\E\left[\sup\limits_{t\in[0,1]}|W_t|^p\right]\leqslant\frac{\eta'}{2}\left(\frac{1}{\eta}\right)^{p/2}p\Gamma\left(\frac{p}{2}\right)
\end{equation}
where $\Gamma(x):=\int_0^{+\infty}e^{-u}u^{x-1}{\rm d}u$.
\end{propo}
\begin{proof}
\begin{align*}
\sup\limits_{t\in[0,1]}|W_t|
=\sup\limits_{t\in[0,1]}\left(\sum_{i=1}^d|W^i_t|^2\right)^{1/2}
=\left(\sup\limits_{t\in[0,1]}\sum_{i=1}^d|W^i_t|^2\right)^{1/2}
\leqslant\left(\sum_{i=1}^d\sup\limits_{t\in[0,1]}|W^i_t|^2\right)^{1/2}
\leqslant\sum_{i=1}^d\sup\limits_{t\in[0,1]}|W^i_t|.
\end{align*}
Therefore for all $x\geqslant0$, we have
$$\PP\left(\sup\limits_{t\in[0,1]}|W_t|\geqslant x\right)\leqslant\PP\left(\sum_{i=1}^d\sup\limits_{t\in[0,1]}|W^i_t|\geqslant x\right)\leqslant\sum_{i=1}^d\PP\left(\sup\limits_{t\in[0,1]}|W^i_t|\geqslant x\right)=d\times\PP\left(\sup\limits_{t\in[0,1]}|W^1_t|\geqslant x\right).$$
Since $\sup\limits_{t\in[0,1]}|W^1_t|=\max\left(\sup\limits_{t\in[0,1]}(-W^1_t),~\sup\limits_{t\in[0,1]}W^1_t\right)~$ and $~(W^1_t)_{t\geqslant0}\overset{\mathcal{L}}{=}(-W^1_t)_{t\geqslant0}$, we have $$\PP\left(\sup\limits_{t\in[0,1]}|W_t|\geqslant x\right)\leqslant d\left(\PP\left(\sup\limits_{t\in[0,1]}(-W^1_t)\geqslant x\right)+\PP\left(\sup\limits_{t\in[0,1]}W^1_t\geqslant x\right)\right)= 2d\times\PP\left(\sup\limits_{t\in[0,1]}W^1_t\geqslant x\right).$$
By the reflection principle, we know that $\PP\left(\sup\limits_{t\in[0,1]}W^1_t\geqslant x\right)=2\PP(W^1_1\geqslant x)$ which induces finally that
\begin{equation}
\PP\left(\sup\limits_{t\in[0,1]}|W_t|\geqslant x\right)\leqslant 4d~\PP(W^1_1\geqslant x)=\frac{4d}{\sqrt{2\pi}}\int_{x}^{+\infty}e^{-\frac{1}{2}s^2}{\rm d}s\leqslant C_d e^{-\frac{1}{4}x^2}.
\end{equation}
Then, \eqref{prop:condition2_sub_gaussian_sup_BM} follows from \eqref{prop:condition1_sub_gaussian_sup_BM} by using the formula $\E[X]=\int_0^{+\infty}\PP(X>x){\rm d}x$ for non-negative random variables and a simple change of variable.
\end{proof}

\section{Uniform sub-Gaussianity of $\|G^{(k)}\|_{\alpha,[0,2]}$}\label{section:uniform_sub_gaussian_G}
In this section, we consider the following Gaussian processes: for all $k\in\N^*$,
\begin{equation}
\forall v\in[0,2],\quad G^{(k)}_v:=\int_{0}^{1\wedge v}K^{}_H(v+k-1,s+k-1){\rm d}W_s
\end{equation}
where $(W_t)_{t\in[0,T]}$ is a $d$-dimensional Brownian motion and $K_H$ is defined by \eqref{eq:kernel_K_H}.
\begin{rem} Since we are interested in the law of $G^{(k)}$, we have replaced $W^{(k)}$ by $W$ in the expression of $G^{(k)}$ given by \eqref{def:G}.
\end{rem}

First, we have the following control on the second moment of $G^{(k)}$-increments.

\begin{propo}\label{prop:moment_2_increments_G}
There exists $C_H>0$ such that for all $k\in\N^*$ and for all $0\leqslant v'<v\leqslant2$,
\begin{equation}
\E\left[\left|G^{(k)}_v-G^{(k)}_{v'}\right|^2\right]\leqslant C_H|v-v'|^{2\alpha^{}_H}
\end{equation}
with $\alpha^{}_H:=\left\{\begin{array}{lll}
H & \text{if} & H<1/2\\
\frac{H}{2} & \text{if} & H>1/2
\end{array}\right.$.
\end{propo}

\begin{proof}Let $0\leqslant v'<v\leqslant2$. Then,
\begin{align}
&G_v^{(k)}-G_{v'}^{(k)}\nonumber\\
&=\int_0^{1\wedge v'}K^{}_H(v+k-1,s+k-1)-K^{}_H(v'+k-1,s+k-1){\rm d}W_s + \int_{1\wedge v'}^{1\wedge v}K^{}_H(v+k-1,s+k-1){\rm d}W_s\nonumber\\
&=\int_0^{1\wedge v'}\left(\int_{v'}^v\frac{\partial}{\partial u}K^{}_H(u+k-1,s+k-1){\rm d}u\right){\rm d}W_s + \int_{1\wedge v'}^{1\wedge v}K^{}_H(v+k-1,s+k-1){\rm d}W_s
\end{align}
with
\begin{equation}\label{eq:first_deriv_Kernel}
\frac{\partial}{\partial u}K^{}_H(u+k-1,s+k-1)=c_H\left(\frac{u+k-1}{s+k-1}\right)^{H-\frac{1}{2}}(u-s)^{H-\frac{3}{2}}.
\end{equation}
Then, we deduce the following expression for the moment of order $2$:
\begin{align}\label{eq:moment_2_increments_G}
&\E\left[|G^{(k)}_v-G^{(k)}_{v'}|^2\right]\nonumber\\
&\quad=\int_0^{1\wedge v'}\left(\int_{v'}^v\frac{\partial}{\partial u}K^{}_H(u+k-1,s+k-1){\rm d}u\right)^2{\rm d}s + \int_{1\wedge v'}^{1\wedge v}K^{}_H(v+k-1,s+k-1)^2{\rm d}s\nonumber\\
&\quad=:I_1(v,v')+I_2(v,v').
\end{align}
Now, let us distinguish the two cases: $k>1$ and $k=1$:\\

$\rhd$ \underline{First case: $k>1$}\\

We begin with the first integral in \eqref{eq:moment_2_increments_G}, namely $I_1(v,v')$: let us note that in the expression \eqref{eq:first_deriv_Kernel}
$$\sup\limits_{k>1}~\sup\limits_{u,s\in[0,2]}\left(\frac{u+k-1}{s+k-1}\right)^{H-\frac{1}{2}}<+\infty.$$
Hence,
\begin{align}\label{eq:I_1_v_v'}
I_1(v,v')&\leqslant C_H\int_0^{1\wedge v'}\left(\int_{v'}^v(u-s)^{H-\frac{3}{2}}{\rm d}u\right)^2{\rm d}s\nonumber\\
&=\frac{C_H}{(H-1/2)^2}\int_0^{1\wedge v'}\left[(v-s)^{H-\frac{1}{2}}-(v'-s)^{H-\frac{1}{2}}\right]^2{\rm d}s\nonumber\\
&\leqslant\frac{C_H}{(H-1/2)^2}\int_0^{v'}\left[(v-s)^{H-\frac{1}{2}}-(v'-s)^{H-\frac{1}{2}}\right]^2{\rm d}s\nonumber\\
&\leqslant C'_H\left\{\begin{array}{lll}
(v-v')^{2H} & \text{if} & H<1/2\\
(v-v')^H & \text{if} & H>1/2
\end{array}\right..
\end{align}
and the last inequality is given by the following estimate:
\begin{lem}There exists $\tilde{C}_H>0$ such that for all $0\leqslant v'< v\leqslant2$,
$$\int_0^{v'}\left[(v-s)^{H-\frac{1}{2}}-(v'-s)^{H-\frac{1}{2}}\right]^2{\rm d}s\leqslant \tilde{C}_H\left\{\begin{array}{lll}
(v-v')^{2H} & \text{if} & H<1/2\\
(v-v')^H & \text{if} & H>1/2
\end{array}\right..$$
\end{lem}
\begin{proof}First, we easily have 
\begin{align*}
\int_0^{v'}\left[(v-s)^{H-\frac{1}{2}}-(v'-s)^{H-\frac{1}{2}}\right]^2{\rm d}s=\frac{1}{2H}\left[v^{2H}+v'^{~2H}-(v-v')^{2H}\right]-2\int_0^{v'}[(v-s)(v'-s)]^{H-\frac{1}{2}}{\rm d}s.
\end{align*}
Now, since
$$[(v-s)(v'-s)]^{H-\frac{1}{2}}\geqslant\left\{\begin{array}{lll}
(v-s)^{2H-1} & \text{if} & H<1/2\\
(v'-s)^{2H-1} & \text{if} & H>1/2
\end{array}\right.$$

we get after some computations
\begin{align*}
\int_0^{v'}\left[(v-s)^{H-\frac{1}{2}}-(v'-s)^{H-\frac{1}{2}}\right]^2{\rm d}s
\leqslant\frac{1}{2H}\left\{\begin{array}{lll}
v'^{~2H}-v^{2H}+(v-v')^{2H} & \text{if} & H<1/2\\
v^{2H}-v'^{~2H}-(v-v')^{2H} & \text{if} & H>1/2
\end{array}\right..
\end{align*}
Moreover, when $H>1/2$, for all $0\leqslant v'<v\leqslant2$,
\begin{align*}
v'^{~2H}-v^{2H}+(v-v')^{2H}&=(v^H-v'^{~H})(v^H+v'^{~H})-(v-v')^{2H}\\
&=(v-v')^H\left(v^H+v'^{~H}-(v-v')^H\right)\\
&\leqslant C_H(v-v')^H
\end{align*}
and when $H<1/2$, $~v'^{~2H}-v^{2H}<0$. So finally, we have the desired result.
\end{proof}
We can now move on the second term in \eqref{eq:moment_2_increments_G}, namely $I_2(v,v')$.
By Theorem 3.2 in \cite{decreusefond1999stochastic}, we have the following upper bound
\begin{align*}
I_2(v,v')\leqslant c_H^2\int_{1\wedge v'}^{1\wedge v}(s+k-1)^{-2\left|H-\frac{1}{2}\right|}(v-s)^{-2\left(\frac{1}{2}-H\right)_+}{\rm d}s
\end{align*}
where $x_+=\max(x,0)$. Then, since $~\sup\limits_{k>1}\sup\limits_{s\in[0,2]}(s+k-1)^{-2\left|H-\frac{1}{2}\right|}<+\infty$, we have 
\begin{align}\label{eq:I_2_v_v'}
I_2(v,v')&\leqslant C_H\int_{1\wedge v'}^{1\wedge v}(v-s)^{-2\left(\frac{1}{2}-H\right)_+}{\rm d}s\nonumber\\
&=C_H(1\wedge v-1\wedge v')^{2\left(H\wedge\frac{1}{2}\right)}\nonumber\\
&\leqslant C_H(v-v')^{2\left(H\wedge\frac{1}{2}\right)}\nonumber\\
&\leqslant \tilde{C}_H\left\{\begin{array}{lll}
(v-v')^{2H} & \text{if} & H<1/2\\
(v-v')^H & \text{if} & H>1/2
\end{array}\right..
\end{align}
By using \eqref{eq:I_1_v_v'} and \eqref{eq:I_2_v_v'} in \eqref{eq:moment_2_increments_G}, we end the proof of Proposition \ref{prop:moment_2_increments_G} for $k>1$.\\

$\rhd$ \underline{Second case: $k=1$}\\

Let us divide this part of the proof into three new cases:\\
First, consider $0\leqslant v'<v\leqslant1$, then $G^{(1)}$ coincides in law with the fractional Brownian motion:
\begin{align*}
\E\left[|G_v^{(1)}-G_{v'}^{(1)}|^2\right]=(v-v')^{2H}.
\end{align*}
Secondly, for $1\leqslant v'<v\leqslant2$, by \eqref{eq:moment_2_increments_G}:
\begin{align*}
\E\left[|G_v^{(1)}-G_{v'}^{(1)}|^2\right]&=\int_0^{1}\left(\int_{v'}^v\frac{\partial}{\partial u}K^{}_H(u+k-1,s+k-1){\rm d}u\right)^2{\rm d}s \\&\leqslant \int_0^{v'}\left(\int_{v'}^v\frac{\partial}{\partial u}K^{}_H(u+k-1,s+k-1){\rm d}u\right)^2{\rm d}s + \int_{v'}^{ v}K^{}_H(v+k-1,s+k-1)^2{\rm d}s\\&\quad\quad=(v-v')^{2H}.
\end{align*}
Finally, if $0\leqslant v'<1\leqslant v\leqslant2$, we get the following by using the two previous cases:
\begin{align*}
\E\left[|G_v^{(1)}-G_{v'}^{(1)}|^2\right]&\leqslant2\E\left[|G_v^{(1)}-G_{1}^{(1)}|^2\right]+2\E\left[|G_1^{(1)}-G_{v'}^{(1)}|^2\right]\\
&\leqslant2\left((v-1)^{2H}+(1-v)^{2H}\right)\\
&\leqslant 4(v-v')^{2H}.
\end{align*}
This inequality concludes the proof of Proposition \ref{prop:moment_2_increments_G} for $k=1$.
\end{proof}
We can now state the result of uniform Sub-Gaussianity:
\begin{propo}\label{prop:uniform_sub_gaussian_G}
There exist $\eta,\eta'>0$ such that for all $k\in\N^*$,
\begin{equation}\label{prop:condition1_sub_gaussian}
\forall x\geqslant0,\quad \PP\left(\|G^{(k)}\|_{\alpha'_H,[0,2]}>x\right)\leqslant \eta'e^{-\eta x^2}
\end{equation}
with $0<\alpha'_H<\alpha_H$ and $\alpha_H$ is defined in Proposition \ref{prop:moment_2_increments_G}.\\
Consequently, for all $k\in\N^*$ and for all $p\geqslant2$,
\begin{equation}\label{prop:condition2_sub_gaussian}
\E\left[\|G^{(k)}\|_{\alpha'_H,[0,2]}^p\right]\leqslant\frac{\eta'}{2}\left(\frac{1}{\eta}\right)^{p/2}p\Gamma\left(\frac{p}{2}\right)
\end{equation}
where $\Gamma(x):=\int_0^{+\infty}e^{-u}u^{x-1}{\rm d}u$.
\end{propo}

\begin{proof}
Let us first note that since $G^{(k)}$ is a centered Gaussian process (for all $k\in\N^*$), there exists $C>0$ such that for all $p\geqslant1$ and for all $0\leqslant v'<v\leqslant2$:
\begin{align*}
\E\left[\left|G^{(k)}_v-G^{(k)}_{v'}\right|^p\right]^{1/p}\leqslant C\sqrt{p}~\E\left[\left|G^{(k)}_v-G^{(k)}_{v'}\right|^2\right]^{1/2}.
\end{align*}
Then, we obtain through Proposition \ref{prop:moment_2_increments_G},
\begin{align}\label{eq:condition1_sub_gaussian}
\forall p\geqslant1,\quad\sup\limits_{0\leqslant v'<v\leqslant2}\frac{\E\left[\left|G^{(k)}_v-G^{(k)}_{v'}\right|^p\right]^{1/p}}{|v-v'|^{\alpha_H}}\leqslant \tilde{C}_H\sqrt{p}.
\end{align}
Now by Theorem A.19 in \cite{friz2010multidimensional}, \eqref{eq:condition1_sub_gaussian} implies that for all $0<\alpha'_H<\alpha_H$, there exists $\eta_1>0$ such that
\begin{align*}
\E\left[\exp\left(\eta_1\|G^{(k)}\|_{\alpha'_H,[0,2]}^2\right)\right]<+\infty
\end{align*}
and by Lemma A.17 in \cite{friz2010multidimensional} (characterization of Gaussian integrability), this condition is equivalent to the existence of $\eta,\eta'>0$ (depending only on $\eta_1$) such that \eqref{prop:condition1_sub_gaussian} is true. \\
Then, \eqref{prop:condition2_sub_gaussian} follows from \eqref{prop:condition1_sub_gaussian} by using the formula $\E[X]=\int_0^{+\infty}\PP(X>x){\rm d}x$ for non-negative random variables and a simple change of variable.
\end{proof}

\section*{Acknowledgements}

I gratefully acknowledge my PhD advisors Fabien Panloup and Laure Coutin for suggesting the problem and for their valuable comments. 

\bibliographystyle{plain}
\bibliography{biblio_article}

\end{document}